\documentclass[12pt]{article}

 \usepackage{amsmath,amssymb,amscd,amsthm,esint}

\usepackage{graphics,amsmath,amssymb,amsthm,mathrsfs}

\usepackage{graphics,amsmath,amssymb,amsthm,mathrsfs,amsfonts}

\usepackage{sidecap}
\usepackage{float}
\usepackage{extarrows}
\usepackage{booktabs}
\usepackage{verbatim}
\usepackage{hyperref}
\usepackage[usenames,dvipsnames]{xcolor}

\newcommand{\supp}[1]{\mbox{supp}\left(#1\right)}

\newtheorem{thm}{Theorem}[section]
\newtheorem{lemma}[thm]{Lemma}

\theoremstyle{definition}
\newtheorem{remark}[thm]{Remark}

\def\Xint#1{\mathchoice
      {\XXint\displaystyle\textstyle{#1}}%
      {\XXint\textstyle\scriptstyle{#1}}%
      {\XXint\scriptstyle\scriptscriptstyle{#1}}%
      {\XXint\scriptscriptstyle\scriptscriptstyle{#1}}%
                   \!\int}
\def\XXint#1#2#3{{\setbox0=\hbox{$#1{#2#3}{\int}$}
         \vcenter{\hbox{$#2#3$}}\kern-.5\wd0}}

\def\dashint{\Xint-}

\def\e{\varepsilon}

\numberwithin{equation}{section}

\begin{document}

\title{Higher Order Elliptic Equations on Nonsmooth Domains}
\author{
Jun Geng}
\date{}
\maketitle

\let\thefootnote\relax\footnotetext{{E-mail}: gengjun@lzu.edu.cn}
\begin{abstract}
In 1995, D. Jerison and C. Kenig in \cite{JK-1995} considered the the inhomogeneous
Dirichlet problem $\Delta u= f$ on $\Omega$, $u=0$ on $\partial\Omega$ in Lipschitz domains. One of their main results shows that the $W^{1,p}$ estimate holds for the sharp range $\frac{3}{2}-\e<p<3+\e$ for $d\geq 3$ and $\frac{4}{3}-\e<p<4+\e$ if $d=2$. Although
the argument employed in \cite{JK-1995} yield optimal results, they rely on an essential fashion on the maximum principle and, as such,
do not readily adapt to higher-order case. By using a new method, the aim of this paper is to establish an extension of their theorem for higher order inhomogeneous elliptic equations on bounded Lipschitz and convex domains, uniform $W^{\ell,p}$ estimates are obtained for $p$ in certain ranges. Especially, compare to the result in \cite{MM-2013} for biharmonic equation, a larger, sharp, range of $p's$ was obtained in this paper.

\medskip

\medskip
\end{abstract}

\tableofcontents

\section{Introduction}\label{section-1}
Unlike the theory of boundary value problems for second order elliptic operators on Lipschitz domains, higher order elliptic boundary value problems, while having abundant application in physics and engineering, have mostly been out of reach of the methods devised to study the second order case. In 1995, D. Jerison and C. Kenig in \cite{JK-1995} considered the the inhomogeneous Dirichlet problem
$$\Delta u= f~~~\text{in}~~\Omega,~~~u=0~~~\text{on}~~~\partial\Omega$$
for the Laplace equation in Lipschitz domains. One of their main results shows that the $W^{1,p}$ estimate holds for the sharp range $\frac{3}{2}-\e<p<3+\e$ for $d\geq 3$ and $\frac{4}{3}-\e<p<4+\e$ if $d=2$.  Although
the argument employed in \cite{JK-1995} yield optimal results, they rely in an essential fashion on the maximum principle and, as such,
do not readily adapt to higher order inhomogeneous elliptic equation(polyharmonic equation and biharmonic equation).

This paper concentrates on higher order inhomogeneous elliptic systems with real constant coefficients. Specifically, we discuss the fundamental a priori estimates on solutions to biharmonic and other higher order differential equations in arbitrary bounded Lipschitz domains.
More precisely, let $\Omega$ be a bounded Lipschitz domain in $\mathbb{R}^d$. Assume $F\in W^{-\ell,p}(\Omega)$. Consider the inhomogeneous elliptic system of order $2\ell$,
\begin{equation}\label{DP}
	\left\{
	\aligned
	\mathcal{L}(D)  u  & =F & \quad & \text{ in } \Omega,\\
	\frac{\partial^{\ell-1} u}{\partial N^{\ell-1}} & =0 & \quad & \text{ on } \partial\Omega,\\
    D^\alpha u& =0 & \quad & \text{ on } \partial\Omega,~~\text{for}~~|\alpha|\leq \ell-2,\\
    u&\in W^{\ell,p}(\Omega)
	\endaligned
	\right.
\end{equation}
where $\frac{\partial^{\ell-1} u}{\partial N^{\ell-1}}=\sum_{|\alpha|=\ell-1}\frac{(\ell-1)!}{\alpha!}N^\alpha D^\alpha u$ and $N$ denotes the outward unit normal to $\partial\Omega$.
\begin{align}\label{elliptic operator}
	&(\mathcal{L}(D)  u)^j =\sum_{k=1}^m \mathcal{L}^{jk}(D)  u^k \quad j=1,2,...,m,\\
& \mathcal{L }^{jk}(D)=\sum_{|\alpha|=|\beta|=\ell} a_{\alpha\beta}^{jk}D^\alpha D^\beta
\end{align}
and $D=(D_1,D_2,...,D_d)$, $D_i=\frac{\partial}{\partial x_i}$ for $i=1,2,...,d$. Also $\alpha=(\alpha_1,\alpha_2,...,\alpha_d)$ is a multi-index with length $|\alpha|=\alpha_1+...+\alpha_d$, and $D^\alpha=D_1^{\alpha_1}D_2^{\alpha_2}...D_d^{\alpha_d}$. Let

\begin{align}\label{}
	\mathcal{L}^{jk}(\xi)=\sum_{|\alpha|=|\beta|=\ell} a_{\alpha\beta}^{jk}\xi^\alpha \xi^\beta\quad{\rm for~~ \xi\in \mathbb{R}^d}.
\end{align}

We assume that $a_{ij}^{\alpha\beta}$ are real constants and satisfy the symmetry condition
\begin{align}\label{symmetric}
	\mathcal{L}^{jk}(\xi)=\mathcal{L}^{kj}(\xi)
\end{align}
and the Legendre-Hadamard ellipticity condition
\begin{equation}\label{ellipticity}
	\mu|\xi|^{2\ell}|\eta|^2\leq \sum_{j,k=1}^m \mathcal{L}^{jk}(\xi)\eta_j\eta_k\leq \frac{1}{\mu}\mu|\xi|^{2\ell}|\eta|^2\quad {\rm for~ some~\mu>0, ~all~\xi\in\mathbb{R}^d, \eta\in \mathbb{R}^m}.
\end{equation}

The Dirichlet problem \eqref{DP} is said uniquely solvable if for any $F\in W^{-\ell,p}(\Omega)$, there exists a unique $u\in W^{\ell,p}(\Omega)$ such that
\begin{equation}\label{weak formula}
	\int_\Omega \sum_{|\alpha|=|\beta|=\ell}a_{ij}^{\alpha\beta}D^\alpha\phi_i D^\beta u_j=\langle F, \phi\rangle\quad {\rm for~~ any~~\phi\in C^\infty(\Omega)}.
\end{equation}
Moreover, the solution $u$ satisfies the $W^{\ell,p}(\Omega)$ estimate
\begin{equation}\label{}
		\|\nabla^\ell u\|_{L^p(\Omega)} \le C \|F\|_{W^{-\ell,p}(\Omega)}.
\end{equation}	

Compare to $C^{1,\alpha}$ domain, as expected for non-smooth domains, the main difficulty lies in the boundary estimates and one may not expect that the boundary $W^{1,p}$ estimates hold for all $1<p<\infty$ in Lipschitz domains. The following are the main results of the paper.

\begin{thm}\label{main-thm-3}
	Let $\Omega$ be a bounded Lipschitz domain in
$\mathbb{R}^d$ with $d\geq 2$. Then the $L^p$ Dirichlet problem \eqref{DP} is uniquely solvable for
	$$
\frac{2d}{d+1}-\e<p<\frac{2d}{d-1}+\e \quad\quad \e>0.
    $$
The ranges of $p$ are sharp for $d=2,3$.
\end{thm}

\begin{thm}\label{main-thm-4}
	Let $\Omega$ be a bounded Lipschitz domain in
$\mathbb{R}^d$ with $d\geq 2$. For the polyharmonic equation $\Delta^\ell u=F$ in $\Omega$, $\ell\geq 3$, the $L^p$ Dirichlet problem is uniquely solvable for
\begin{equation*}
	\left\{
	\aligned
	4/3-\e<p< 4+\e & \qquad\text{ if } d=2,\\
     3/2-\e<p< 3+\e & \qquad\text{ if } d=3,\\
	\frac{2d}{d+1}-\e<p<\frac{2d}{d-1}+\e & \qquad\text{ if } 4\leq d\leq 2\ell+1 \quad \text{or} \quad d\geq 2\ell+3,\\
    2-\frac{1}{\ell+1}-\e<p< 2+\frac{1}{\ell}+\e &  \qquad\text{ if } d=2\ell+2.
	\endaligned
	\right.
\end{equation*}
The ranges of $p$ are sharp for $2\leq d\leq 2\ell+2$.
\end{thm}

In the next theorem, set $$\lambda_d=\frac{d+10+2(2(d^2-d+2))^{1/2}}{7}.$$

\begin{thm}\label{main-thm-5}
	Let $\Omega$ be a bounded Lipschitz domain in
$\mathbb{R}^d$ with $d\geq 2$. For the biharmonic equation $\Delta^2 u=F$ in $\Omega$, the $L^p$ Dirichlet problem is uniquely solvable for
\begin{equation*}
	\left\{
	\aligned
	4/3-\e<p< 4+\e & \qquad\text{ if } d=2,\\
   3/2-\e<p< 3+\e & \qquad\text{ if } d=3,\\
     8/5-\e<p< 8/3+\e & \qquad\text{ if } d=4,\\
	5/3-\e<p< 5/2+\e & \qquad\text{ if } d=5,6,7,\\
    2-\frac{2}{d-\lambda_d+4}-\e<p< 2+\frac{2}{d-\lambda_d+2}+\e &  \qquad\text{ if } d\geq 8.
	\endaligned
	\right.
\end{equation*}
The ranges of $p$ are sharp for $d=2,3,4,5,6$ and $7$. Moreover, if $\Omega$ is a bounded convex domain in
$\mathbb{R}^d$ with $d\geq 2$, then the $L^p$ Dirichlet problem for $\Delta^2 u=F$ is uniquely solvable for
	$$
1<p<\infty.
    $$
\end{thm}

\begin{remark}
In \cite{MM-2013}, I. Mitrea and M. Mitrea investigated the inhomogeneous Dirichlet problem for the biharmonic equation with data in Besov and Triebel-Lizorkin spaces and sharp ranges of $p,s$ are obtained when $d=4,5$. In this paper, a larger, sharp, range of $p's$ was obtained, we show that the ranges of $p$ are sharp even for $d=6,7$.
\end{remark}

\begin{remark}
Although some progress has been made in recent years, there remain many basic open questions for higher order equations, even in the case of constant coefficient operators in Lipschitz domains. For example, \\
1) for $d\geq 2\ell+3$ and $\ell\geq 3$, what is the sharp range of $p$ for which the $L^p$ Dirichlet problem is uniquely solvable?\\
2) for $d\geq 8$ and the biharmonic equation $\Delta^2 u=F$ on $\Omega$ in $\mathbb{R}^d$, what is the sharp range of $p$ for which the $L^p$ Dirichlet problem is uniquely solvable?
\end{remark}

We remark that for Laplace's equation in Lipschitz domain, the Dirichlet problem with source term in $L^p$ is well understood. Indeed, it has been established by Kenig and Jerison\cite{JK-1995} since the early 1990's that the Dirichlet problem $\Delta u=\text{div}f$ with optimal estimate $\|\nabla u\|_{L^p}\leq C\|f\|_{L^p}$ is uniquely solvable for $\frac{3}{2}-\e<p<3+\e$ for $d\geq 3$ and $\frac{4}{3}-\e<p<4+\e$ if $d=2$ where $\e>0$ depends on $d$ and the Lipschitz character of $\Omega$. This work has been extended to second order elliptic equation with $VMO$ coefficients by Shen \cite{Shen-2005}; the bounds of the Riesz transform in $L^p$ are also obtained there. The result was also extended in \cite{AQ-2002} to the case of $C^1$ domains, for operators with complex coefficients. In \cite{Geng-2012}, the author considered the $W^{1,p}$ estimate for second order elliptic systems subjected to Neumann boundary condition, as a result, under the assumption that $A\in VMO$, we proved that the $W^{1,p}$ estimate holds for $\frac{3}{2}-\e<p<3+\e$ for $d\geq 3$ and $\frac{4}{3}-\e<p<4+\e$ if $d=2$.

In \cite{AP-1998}, Adolfsson and Pipher investigated the inhomogeneous Dirichlet problem for the biharmonic equation with data in Besov and Sobolev spaces. They showed that if $\dot{f}\in WA_{1+s}^p(\partial\Omega)$ and $F\in L_{s+1/p-3}^p(\Omega)$, then there exists a unique function $u$ satisfying $\Delta^2 u=F$ in $\Omega$ and ${\rm Tr}(\partial^\ell u)=f_\ell$ on $\partial\Omega$ for $0\leq |\ell|\leq 1$ subjected to the estimate
\begin{equation}\label{Besov}
\|u\|_{L_{s+1/p+1}^p(\Omega)}\leq C\|F\|_{L_{s+1/p-3}^p(\Omega)}+C\|\dot{f}\|_{WA_{1+s}^p(\partial\Omega)}~~~\text{for} ~~2-\e<p<2+\e,
\end{equation}
here $0<s<1$ and ${\rm Tr}(\partial^\ell u)$ denotes the trace of $\partial^\ell u$ in the sense of Sobolev spaces. In the case of Lipschitz domain in $\mathbb{R}^3$, the range of $p$ is better, precisely, they proved
$\max(1,\frac{2}{s+1+\e})<p< \infty$ if $s<\e$ and $\max(1,\frac{2}{s+1+\e})<p< \frac{2}{s-\e}$ if $\e\leq s<1$. In \cite{MMW-2011}, I. Mitrea, M. Mitrea and Wright extended the $3-$dimensional results to $p=\infty$ if $0<s<\e$ or $\frac{2}{s+1+\e}<p\leq 1$ if $1-\e<s<1$. Later in \cite{MM-2013-1} I.Mitrea and M. Mitrea extended the results of \cite{AP-1998} to higher dimension, they showed that the estimate \eqref{Besov} holds for $\max(1,\frac{d-1}{s+(d-1)/2+\e})<p< \infty$ if $(d-3)/2+s<\e$ and $\max(1,\frac{d-1}{s+(d-1)/2+\e})< p < \frac{d-1}{(d-3)/2+s-\e}$ if $\e\leq s<1$.
More results on the inhomogeneous Dirichlet problem for $\Delta^2$ as well as for general higher order elliptic systems
may be found in\cite{MM-2013}\cite{MMS-2010}\cite{Shen-2007}.

Another important direction concentrates on the $L^p$ boundary value problem of the higher order
elliptic problems.
For the general elliptic equations and system $\mathcal{L}(D)u=0$, with boundary data in $L^p$, the solvability of $L^p$ Dirichlet problem has been established for $$2-\e<p<2+\e,$$ specifically, see \cite{DKV-1988} for second order elliptic systems, and \cite{DKV-1986}\cite{Verchota-1987}\cite{Verchota-1990} for the biharmonic equations and polyharmonic equations, as well as \cite{PV-1995} for general higher order elliptic equations and systems. With the $L^2$ solvability established, one significant problem is to determine the sharp ranges of $p$ for which the Dirichlet problem for strongly elliptic systems with $L^p$
data is well-posed.

Let us  summarize here the results currently known for the ranges of $p$. For $\mathcal{L}(D)u=0$, we have that
\begin{itemize}
  \item
{{$2-\e<p\leq\infty$}}, \quad\qquad\qquad\qquad\text{ if } $d=2,3$, \qquad\cite{PV-1995}\cite{MM-2013}
\item
{{$2-\e<p<\frac{2(d-1)}{d-3}+\e$}}, \qquad\qquad\text{ if } $d\geq 4$, \qquad\qquad\cite{Shen-2006}.
\end{itemize}
We mention that in the lower dimensional case $d=2,3$, the ranges of $p$ are sharp since it is known that the $L^p$ Dirichlet problem for Laplace's equation is uniquely solvable for $2-\e<p\leq \infty$.

Note that if $d\geq 3$, the polyharmonic operator $(-\Delta)^\ell$ may no longer serve as a prototype for the general
elliptic operator $\mathcal{L}(D)  u $. It was proved in \cite{Shen-2006-MA} the polyharmonic equation $\Delta^\ell u=0(\ell\geq 2)$ in $\Omega$ with $D^\alpha u=f_\alpha$ on $\partial\Omega$ for $|\alpha|\leq \ell-2$ and $\partial_\nu^{\ell-1} u=g$ on $\partial\Omega$ is solvable for
 \begin{itemize}
  \item
{{$2-\e<p\leq\infty$}}, \quad\qquad\qquad\qquad\text{ if }  ~~$d=2,3$, \quad\qquad\qquad\qquad~\cite{PV-1995,PV-1993,PV-1992}
\item
{{$2-\e<p<\frac{2(d-1)}{d-3}+\e$}}, \qquad\qquad\text{ if } ~~$4\leq d\leq 2\ell+1$ , \qquad\qquad\cite{Shen-2006,PV-1995-1}
\item
{{$\frac{2(d-1)}{d-3}-\e<p<\frac{2(d-1)}{d-3}+\e$}}, \qquad\text{ if } \quad $d\geq 2\ell+3$, \qquad\qquad\qquad\cite{Shen-2006,PV-1995-1}
\item
{{$2-\e<p< \frac{2\ell}{\ell-1}+\e$}}, \quad\qquad\qquad\text{ if } \quad$d=2\ell+2$, \qquad\qquad\qquad\cite{Shen-2006-MA}
\item
{{Not solvable}}, \qquad\qquad\qquad\qquad\text{ if } ~~$4\leq d\leq 2\ell+1$~~ and ~~$p>\frac{2(d-1)}{d-3}$ , \qquad\qquad
\item
{{Not solvable}}, \qquad\qquad\qquad\qquad\text{ if } ~~$d\geq 2\ell+2$~~ \qquad and ~~$p>\frac{2\ell}{\ell-1}$ , \qquad\qquad
\end{itemize}
and the ranges of $p$ are sharp for $2\leq d\leq 2\ell+2$. In the special case of biharmonic equations, more is known, as follows,
\begin{itemize}
  \item
{{$2-\e<p\leq 6+\e$}}, \quad\qquad\qquad\qquad\text{ if }~~~~$d=4$, \quad\qquad\qquad\qquad~~\cite{Shen-2006,PV-1995-1}
\item
{{$2-\e<p<4+\e$}}, \qquad\qquad\qquad\quad\text{ if }~~~~$d= 5, 6, 7$ , \quad\qquad\qquad~\cite{Shen-2006,PV-1995-1}
\item
{{$2-\e<p< 2+\frac{4}{d-\lambda_d}+\e $}}, \qquad\qquad\text{ if } \quad $d\geq 8$, \quad\qquad\qquad\qquad~\cite{Shen-2006-JGA}
\item
{{$1<p\leq \infty$}}, \qquad\qquad\qquad\qquad\qquad\text{ if } \quad $d\geq 2$,~~$\Omega$ is $C^1$, \quad\quad\quad~~\cite{PV-1993,Verchota-1987}
\item
{{$1<p< \infty$}}, \qquad\qquad\qquad\qquad\qquad\text{ if } \quad $d\geq 2$,~ $\Omega$ is convex, \quad\quad\cite{KS-2011-MA}
\end{itemize}
We refer the reader to \cite{BHM-2018}\cite{BHM-2019} and \cite{BHM-2021} for more results concerned with higher order
elliptic systems on Lipschitz graph or $\mathbb{R}_+^d$.

We also initiate the study of the $W^{\ell,p}$ well-posedness of the Dirichlet problem on convex domains. Note that any convex domain is Lipschitz, but may not be $C^1$. It was obtained in \cite{Shen-2006-JGA} that the biharmonic equation in convex domains is uniquely solvable for $2\leq p<\infty$. Later, the boundedness of $\nabla^2 u$ for biharmonic function in convex domains was established in \cite{Mazya-Mayboroda-2008}, with this, the weak maximum principle $\|\nabla u\|_{L^\infty(\Omega)}\leq C\|\nabla u\|_{L^\infty(\partial\Omega)}$ was established in \cite{KS-2011-MA}, and the solvability of $L^p$ regularity and Dirichlet problems were also obtained there for any $1<p<\infty$. We also remark here in the case of Laplace equation $\Delta u=0$ in $\Omega$ subjected to Neumann boundary condition $\frac{\partial u}{\partial N}=g$ on $\partial \Omega$, the author of this paper and Z. Shen in \cite{Geng-2010} established the nontangential maximal function estimate $\|(\nabla u)^*\|_{L^p(\partial\Omega)}\leq C\|g\|_{L^p(\partial\Omega)}$ for any $1<p<\infty$.

We now describe some of the key ideas in the proof of Theorem \ref{main-thm-3}-Theorem \ref{main-thm-5}. As indicated earlier, our main results continue the work in \cite{JK-1995}\cite{Shen-2005} and study higher order elliptic equations and systems. We present two approaches to prove Theorem \ref{main-thm-3}, one is to reduce the $W^{\ell,p}$ estimates to the weak reverse H\"older inequality of solution to $\mathcal{L}(D)v=0$ in $D_{4r}$ and $D^\alpha v=0$ on $\Delta_{4r}$. The key ingredients in this approach is the following `Good-$\lambda$' inequality
$$
|E(\theta^{-\frac{2}{q}}t)|\leq (\theta/2)|E(t)|+|\{x\in Q_0:\mathcal{M}_{4B_0}(|f\chi_{4B_0\cap \Omega}|^2)(x)\geq \gamma t\}|,
$$
while this could be done by a refined Calder\'on-Zygmund decomposition, we refer the reader to see Lemma 2.2. With this, the $W^{\ell,p}$ solvability of the higher order equation \eqref{DP} will be reduced to the weak reverse H\"older inequality \eqref{WRH}. Finally, together with the interior estimate \eqref{interior} for solutions of $\mathcal{L}(D)  u=0$,
and the regularity estimate
$$
\|(\nabla^\ell u)^*\|_{L^{2}(\partial \Omega)}\leqslant C\|\nabla_{tan}\nabla^{\ell-1}u\|_{L^{2}(\partial \Omega)}
$$
as well as the $L^2$ solvability of the Dirichlet problem
$$
\|(\nabla^{\ell-1} u)^*\|_{L^{2}(\partial \Omega)}\leq C\|\nabla^{\ell-1}u\|_{L^{2}(\partial \Omega)},
$$
where $(\nabla^\ell u)^*$ and $(\nabla^{\ell-1} u)^*$ denote the standard non-tangential maximal function of $\nabla^\ell u$ and $\nabla^{\ell-1} u$ respectively(see the precise definitions in the end of this section), we are able to show that the $W^{\ell,p}$ estimate holds for certain ranges of $p$.
Another approach in the  proof of Theorem \ref{main-thm-3} utilizes the square function estimate
\begin{align}\label{square function}
		&\|S(\nabla^{\ell-1}u)\|_{L^p(\partial\Omega)}\leq C\|(\nabla^{\ell-1}u)^*\|_{L^p(\partial\Omega)}\\
&\|(\nabla^{\ell-1}u)^*\|_{L^p(\partial\Omega)}\leq C\|S(\nabla^{\ell-1}u)\|_{L^p(\partial\Omega)}+C|\nabla^{\ell-1}u(P_0)||\partial\Omega|^{1/p}.
\end{align}
to obtain a much simpler and stronger sufficient condition
$$
\int_{D_r} |\nabla^\ell v|^2\, dx
		\le C \bigg(\frac{r}{R}\bigg)^{\lambda-2}\int_{D_{R}}
		|\nabla^\ell v|^2\, dx.
$$
With this, we are able to show the $L^p$ Dirichlet problem \eqref{DP} is solvable for $2<p<2+\frac{2}{d+2-\lambda}.$

Our methods in the proof of Theorems \ref{main-thm-5} is different from those in Theorem \ref{main-thm-3} and Theorem \ref{main-thm-4} as well as from those used in the case of Laplacian. Our approach is originally motivated by the work of Mazya \cite{Mazya-1979} and further developed by Shen \cite{Shen-2006-JGA}. One of the significant novelties of our argument in the proof of Theorem \ref{main-thm-5} is to obtain new weighted positivity \eqref{4.17} in Lipschitz domains, especially in the case of $d\geq 8$:
$$
\int_\Omega u_{ii} (u |x-y|^{-\sigma})_{ii}dx\geq \frac{1}{4}\bigg(d^2+2d\sigma-7\sigma^2-8\sigma\bigg)\int_\Omega \bigg|\frac{\partial u}{\partial |x-y|}\bigg|^2|x-y|^{-2-\sigma},
$$
since it is known that the solution of Dirichlet problem for the polyharmonic equation is continuous at the vertex of a conic domain if the operator satisfies the Mazya-type inequality
\begin{align}\label{Mazya-type}
		\int_\Omega(\Delta^2 u) u |x|^{4-d}dx\, >0.
\end{align}
However, \eqref{Mazya-type} only holds for $d=5,6,7$, but not for $d\geq 8$.

In the proof of Theorem \ref{main-thm-5} in convex domain, sharp ranges of $p$ are obtained. One may notice that any convex domain is Lipschitz, but may not be $C^1$. To obtain the sharp ranges, we return to Theorem \ref{main-thm-1}, i.e.,
it suffices for us to show that
estimate \eqref{WRH} holds for any $p>2$. Let $\Omega$ be convex domain in $\mathbb{R}^d, d\geq 2$ and $u$ a weak solution of $\Delta^2 u=0$ in $\Omega$ and $u=\frac{\partial u}{\partial N}=0$ on $\Delta_{5r}$.
It follows from the boundedness of $\nabla^2 u$ for biharmonic function in convex domains in \cite{Mazya-Mayboroda-2008},
 gives that
\begin{equation}
\label{reverseHolder}
\left\{ \frac{1}{r^{d}}
\int_{B(P,3r)\cap \partial\Omega}
|\nabla^2 u|^p\, d\sigma\right\}^{1/p }
\le C \left\{ \frac{1}{r^{d}}
\int_{B(P,3r)\cap \partial\Omega}
|\nabla^2 u|^2\, d\sigma\right\}^{1/2}.
\end{equation}
This gives estimate \eqref{WRH} immediately.

In the end of this section, we establish some terminology. Let $\psi: \mathbb{R}^{n-1}\rightarrow \mathbb{R}$ be a Lipschitz mapping and $\psi(0)=0$.
Recall that a bounded domain $\Omega$ in $\mathbb{R}^d$ is called a Lipschitz domain if there exists
$r_0>0$ such that for each point $Q\in \partial\Omega$, there is a new coordinate system of $\mathbb{R}^d$,
obtained from the standard Euclidean coordinate system through translation and
rotation, so that $Q=(0,0)$ and
\begin{equation}\label{2.2}
B(Q,r_0)\cap \Omega=\{(x^{\prime}, x_d)\in \mathbb{R}^d: x_d> \psi(x^{\prime})\}\cap B(Q,r_0),
\end{equation}
\begin{equation}\label{2.3}
B(Q,r_0)\cap \partial \Omega=\{(x^{\prime}, x_d)\in \mathbb{R}^d: x_d= \psi(x^{\prime})\}\cap B(Q,r_0),
\end{equation}
where $B(Q,r_0)=\{x\in\mathbb{R}^d:|x-Q|<r_0\}$ denotes a ball centered at $Q$ with radius $r_0$. For $r>0$, set
\begin{equation}\label{2.4}
\Delta_r=\{(x^{\prime}, \psi(x^\prime))\in \mathbb{R}^d: |x^\prime|<r\},
\end{equation}
\begin{equation}\label{2.5}
~~~~~~~~~~~~~~~D_r=\{(x^{\prime}, x_d)\in \mathbb{R}^d: |x^\prime|<r~~\text{and}~~\psi(x^\prime)<x_d<\psi(x^\prime)+(M+10d)r\},
\end{equation}
denote the Lipschitz cylinder and its surface.
For simplicity, we use $u_i$, $u_{ij}$ to denote $\frac{\partial u}{\partial x_i}$, $\frac{\partial^2u}{\partial x_i\partial x_j}$ respectively. Finally, throughout the paper we will use $\dashint_{\Omega}u=\frac{1}{|\Omega|}\int_\Omega u$ to denote the average of $u$ over $\Omega$, and $\|f\|_p$ is defined to mean the norm of $f$ in $L^p{(\partial\Omega)}$ for $1<p<\infty$. Also, $(\nabla^\ell u)^*$ and $(\nabla^{\ell-1}u)^*\in L^p(\partial\Omega)$ is defined to mean the nontangential maximal function of $\nabla^\ell u$ and $\nabla^{\ell-1}u$ respectively, defined by
\begin{equation}
(\nabla^{\ell}u)^*(Q)=\sup\{|\nabla^{\ell} u(x)|:x\in\Omega\quad\text{and}\quad |x-Q|<C\delta (x)\}
\end{equation}
and
\begin{equation}
(\nabla^{\ell-1}u)^*(Q)=\sup\{|\nabla^{\ell-1}u(x)|:x\in\Omega\quad\text{and}\quad |x-Q|<C\delta (x)\}
\end{equation}
for $Q\in\partial\Omega$ where $\delta(x)=\text{dist}(x,\partial\Omega)$ and $C$ is large enough.

\section{Sufficiency}

\subsection{A sufficient condition}

\begin{thm}\label{main-thm-1}
	Let $\mathcal{L}(D)$ be a system of elliptic operators of order $2\ell$ given by \eqref{elliptic operator} and $\Omega$  a bounded Lipschitz domain.
	Assume that  $A$ is a matrix satisfying \eqref{symmetric} and \eqref{ellipticity}. Let $2<p<\infty$. Suppose that for any $|\alpha|\leq \ell-1$ and $B=B(x_0,r)$, where $x_0\in\partial\Omega$ and $0<r<c_0{\rm diam}(\Omega)$, the weak reverse H\"older inequality
\begin{equation}\label{WRH}
		\bigg(\fint_{D_r}
		|\nabla^\ell v|^p\, dx\bigg)^{1/p}
		\le C_0 \bigg(\fint_{D_{2r}}
		|\nabla^\ell v|^2\, dx\bigg)^{1/2}
\end{equation}
 holds for solutions of
\begin{equation}\label{local-equation}
		\mathcal{L}(D)v=0 \quad \text{in}\quad D_{4r} ~~\text{and} \\
		~~~~~D^\alpha v=0 \quad \text{on}~~\Delta_{4r}.
	\end{equation}
 Let $u\in H^\ell(\Omega)$ be a weak solution of \eqref{DP} with $F\in W^{-\ell,p}(\Omega)$. Then $u\in W^{\ell,p}(\Omega)$ and
	\begin{equation}\label{w-L2-ineq}
		\|\nabla^\ell u\|_{L^p(\Omega)} \le C \|F\|_{W^{-\ell,p}(\Omega)},
	\end{equation}	
with constant $C$ depends only on $d,p,\ell, \mu, C_0$ and the Lipschitz character of $\Omega$.
\end{thm}

To formulate the `Good-$\lambda$' type estimate, we notice that for any functional $F\in W^{-\ell,p}(\Omega)$, there exist an array of functions $h^\alpha\in L^p(\Omega)$ where $\alpha$ is a multi-index such that
$$
F=\sum_{|\alpha|\leq \ell} D^\alpha h^\alpha,\quad \text{and}\quad \|F\|_{W^{-\ell,p}(\Omega)}\approx \sum_{|\alpha|\leq \ell} \|h^\alpha\|_{L^p(\Omega)},
$$
and in the next lemma, set
$$
f(x)=\sum_{|\alpha|\leq \ell}|h^\alpha(x)|.
$$
and for cube $Q$ on $\partial \Omega$ and a function $f$ defined on $Q$, we define a localized
Hardy-Littlewood maximal function $\mathcal{M}_{Q}$ by
$$
\mathcal{M}_{Q}(f)(P)=\sup_{P\in Q^\prime\subset Q}\dashint_{Q^\prime}|f|
$$

\begin{lemma}\label{Good-lambda-1.2}
Assume that $f$ is defined as above and $f\in L^q(\Omega)$ for some $2<q<p$. Let $u$ be a weak solution to (\ref{DP}) such that \eqref{WRH} holds. Let $B_0=B(x_0,r_0)$ be a ball with the properties that $x_0\in\partial\Omega$ and $0<r_0=c_0\text{diam}(\Omega)$. Let $Q_0$ be a cube such that $2Q_0\subset 2B_0$ and $|Q_0|\approx|B_0|$. Then there exist constants $\theta$, $\gamma\in (0,2^{-d})$, and $C_0>1$, such that
\begin{equation}\label{good-lambda-1.2}
|E(\theta^{-\frac{2}{q}}t)|\leq (\theta/2)|E(t)|+|\{x\in Q_0:\mathcal{M}_{4B_0}(|f\chi_{4B_0\cap \Omega}|^2)(x)>\gamma t\}|
\end{equation}
for any $t\geq t_0=C_0\fint_{4B_0}|\nabla^\ell u\chi_{4B_0\cap\Omega}|^2\, dx$, where
\begin{equation*}
	E(t)=\{x\in Q_0:\mathcal{M}_{4B_0}(|\nabla^\ell u\chi_{4B_0\cap \Omega}|^2)(x)> t\}.
\end{equation*}
\end{lemma}
\begin{proof}
Let $u$ be a weak solution to (\ref{DP}) and fix $x_0\in\partial\Omega, 0<r_0=c_0\text{diam}(\Omega)$. For any ball $B=B(y_0,r)$ such that $|B|\leq c_1r_0^d$ and either $y_0\in B(x_0, 2r_0)\cap\partial\Omega$ or $4B\subset B(x_0, 2r_0)\cap\Omega$, let $v_B\in H^\ell(\Omega)$ be the weak solution of
	\begin{equation}\label{v_B}
		\mathcal{L}(D)  v_B=\sum_{|\alpha|\leq \ell} D^\alpha (h^\alpha \varphi) \quad \text{in}\quad \Omega ~~\text{and} \\
		~~~~~D^\alpha v_B=0 \quad \text{on}~~\partial\Omega,
	\end{equation}
	for all
	$|\alpha|\leq \ell-1$ and $\varphi\in C_0^\infty(\mathbb{R}^d)$ is a cut-off function satisfying
	\begin{equation}
		\begin{cases}
			\varphi=1 \quad \text{ in } 4B,\\
			\varphi=0 \quad \text{ outside } 5B,\\
			0\le \varphi\le 1.
		\end{cases}
	\end{equation}
	By the energy estimate, we have that
	\begin{equation}\label{myers-5}
		\int_{\Omega} |\nabla^\ell v_B|^2
		\le C \int_{\Omega} |\varphi f|^2
		\le C\int_{5B} | f\chi_{4B_0\cap \Omega}|^2.
	\end{equation}
	Since
	\begin{equation*}
		\mathcal{L}(D)  (u-v_B)=0 \quad \text{in}\quad 4B\cap\Omega \quad\text{and}
	\quad D^\alpha(u-v_B)=0 \quad \text{on}~~4B\cap\partial\Omega,
	\end{equation*}
	it follows from \eqref{WRH} that
\begin{equation}\label{3.1-5}
	\aligned
	\left(\fint_{ 2B\cap \Omega}
	|\nabla^\ell (u-v_B)|^{p} \right)^{1/p}
	\le C \left(\fint_{4B\cap \Omega}
		|\nabla^\ell (u-v_B)|^2 \right)^{1/2}
	\endaligned
\end{equation}
	for some $p>2$.

In view of \eqref{myers-5}, this implies that
\begin{equation}\label{u-vB}
	\aligned
		&\left(\fint_{ 2B\cap \Omega}
	|\nabla^\ell (u-v_B)|^{p} \right)^{1/p}\\
	&\qquad\qquad\le C \left\{
	\left(\fint_{ 4B\cap \Omega}
	|\nabla^\ell u|^2
	\right)^{1/2}
	+
	\left(\fint_{4B\cap \Omega }
	|\nabla^\ell v_B|^2 \right)^{1/2}
	\right\}\\
	&\qquad\qquad
	\le
	C \left\{
	\left(\fint_{ 4B}
	|\nabla^\ell u\chi_{4B_0\cap\Omega}|^2
	\right)^{1/2}
	+
	\left(\fint_{5B}
	|f\chi_{4B_0\cap\Omega}|^2 \right)^{1/2}
	\right\}.
	\endaligned
\end{equation}
	
To proceed, by the weak $(1,1)$ estimate for $\mathcal{M}_{4B_0}$,
\begin{equation}
	|E(t)|\leq \frac{C}{t}\int_{4B_0}|\nabla^\ell u\chi_{4B_0\cap \Omega}|^2\,dx\leq \frac{\theta}{2}|Q_0|
\end{equation}
for $t\geq t_0$, where $t_0=C_0\fint_{4B_0}|\nabla^\ell u \chi_{4B_0\cap \Omega}|^2\, dx$ with $C_0=2C|4B_0|/(\theta|Q_0|)$ and $\theta$ is a small constant to be chosen later. Since $E(t)$ is relatively open in $Q_0$, by the Calder\'on-Zygmund decomposition, there exists a sequence of disjoint maximal dyadic subcubes $\{Q_k\}$ of $Q_0$ such that
\begin{equation*}
 Q_k\subset E(t)\qquad\text{and}\qquad|E(t)\setminus\cup_k Q_k|=0.
\end{equation*}
By choosing $\theta$ small, we may assume that $|Q_k|\leq c_1|Q_0|$. For each $Q_k$, let $B_k$ be the smallest ball such that $Q_k\subset B_k$.
We show that there exist constants $\theta$, $\gamma\in (0,2^{-d})$, such that for any $Q_k$ satisfying
\begin{equation}\label{cond-Q-k}
	Q_k\cap\{x\in Q_0:\mathcal{M}_{4B_0}(|f\chi_{4B_0\cap\Omega}|^2)(x)\leq \gamma t\}\neq\emptyset,	
\end{equation}
we have
\begin{equation}\label{omegaE}
	|E(\theta^{-\frac{2}{q}} t)\cap Q_k|	\leq(\theta/2)|Q_k|.
\end{equation}
Estimate (\ref{good-lambda-1.2}) follows from (\ref{omegaE}) directly.

To show (\ref{omegaE}), note that \begin{equation}\label{estimate-maximal}
	\mathcal{M}_{4B_0}(|\nabla^\ell u\chi_{4B_0\cap\Omega}|^2)(x)\leq \max\{\mathcal{M}_{2B_k}(|\nabla^\ell u\chi_{4B_0\cap\Omega}|^2)(x),C_d t\}
\end{equation}
for any $x\in Q_k$. It follows that if $\theta^{-\frac{2}{q}}\geq C_d$, then
\begin{equation}\label{estimate-omega-E}
\aligned
|E(\theta^{-\frac{2}{q}} t)\cap Q_k|
& \le \left|
\left \{ x\in Q_k:\  \mathcal{M}_{2B_k} (|\nabla^\ell u\chi_{4B_0\cap\Omega}|^2 ) (x) >\theta^{-\frac{2}{q}}  t\right\}\right|.
\endaligned
\end{equation}

We consider the following three cases.

Case 1. If $4B_k\cap\Omega=\emptyset$, then $\mathcal{M}_{2B_k} (|\nabla^\ell u\chi_{4B_0\cap\Omega}|^2)=0$, thus $| (E(\theta^{-\frac{2}{q}} t)\cap Q_k|=0$.

Case 2. If $4B_k\subset\Omega$, then $\mathcal{M}_{2B_k} (|\nabla^\ell u\chi_{4B_0\cap\Omega}|^2)=\mathcal{M}_{2B_k} (|\nabla^\ell u|^2)$. Let $v_{B_k}$ be given by (\ref{v_B}), then $|\nabla^\ell u|\leq|\nabla^\ell v_{B_k}|+|\nabla^\ell (u-v_{B_k})|$. This, together with (\ref{estimate-omega-E}), implies that
\begin{equation}\label{estimate-omegaE-1}
	\aligned
	|E(\theta^{-\frac{2}{q}} t)\cap Q_k|
	&\le \left|\left\{ x\in Q_k: \  \mathcal{M}_{2B_k}
	(|\nabla^\ell v_{B_k} |^2) (x)>\frac{\theta^{-\frac{2}{q}} t }{4} \right\}\right|\\
	& \quad
	+ \left|\left\{ x\in Q_k:  \  \mathcal{M}_{2B_k}
	(|\nabla^\ell(u-v_{B_k}) |^2) (x)>\frac{\theta^{-\frac{2}{q}}
		t }{4} \right\}\right|\\
	&=I_1 +I_2.
	\endaligned
\end{equation}
By the weak (1,1) estimate for $\mathcal{M}_{2B_k}$, we obtain
\begin{equation}\label{estimate-1}
	\aligned
	I_1
	&\leq\frac{C\theta^{\frac{2}{q}}}{t}\fint_{2B_k}|\nabla^\ell v_{B_k}|^2\,dx |Q_k|\\
	&\leq\frac{C\theta^{\frac{2}{q}}}{t}\fint_{5B_k}|f\chi_{4B_0\cap\Omega}|^2\,dx |Q_k|
		\\
&	\leq C\gamma\theta^{\frac{2}{q}}|Q_k|
	\endaligned
\end{equation}
for any $t\geq t_0$, where we have used (\ref{myers-5}) for the second inequality and (\ref{cond-Q-k}) for the last.

We now estimate $I_2$. By the weak $(\frac{p}{2},\frac{p}{2})$ estimate for $\mathcal{M}_{2B_k}$,
\begin{equation*}\label{I2-1.2}
	\aligned
		I_2 &\le  C
	\bigg(
	\frac{\theta^{\frac{2}{q}}}{t}
	\bigg)^{\frac{p}{2}}
	\fint_{2B_k}
	|\nabla^\ell(u-v_{B_k})|^{p}\,  dx\,  |Q_k|\\
	&\le
	C\theta^{\frac{p}{q}}t^{-\frac{p}{2}}
	\left\{\left(\fint_{4B_k}|\nabla^\ell u\chi_{4B_0\cap\Omega}|^2\right)^{{\frac{p}{2}}}+\left(\fint_{5B_k}|f\chi_{4B_0\cap\Omega}|^2\right)^{{\frac{p}{2}}}\right\}|Q_k|\\
&\le
	C\theta^{\frac{p}{q}}t^{-\frac{p}{2}}\left\{t^{p/2}+(t\gamma)^{p/2}\right\}|Q_k|,
	\endaligned
\end{equation*}
where we have used (\ref{u-vB}) for the second inequality and (\ref{cond-Q-k}) for the last.
This, together with (\ref{estimate-omegaE-1}) and (\ref{estimate-1}), gives
\begin{equation*}
|E(\theta^{-\frac{2}{q}} t)\cap Q_k|	\leq \theta\left\{C\gamma\theta^{\frac{2}{q}-1}+C\theta^{\frac{p}{q}-1}\right\}|Q_k|
\end{equation*}
for any $Q_k$ satisfies (\ref{cond-Q-k}). Since $p>q>2$, we can choose $\theta$ then $\gamma>0$ sufficiently small such that $\left\{C\gamma\theta^{\frac{2}{q}-1}+C\theta^{\frac{p}{q}-1}\right\}\leq 1/4$. This gives (\ref{omegaE}).

Case 3. If $4B_k\cap\partial\Omega\neq\emptyset$, suppose that $B_k=B(z_k,r_k)$. Let $y_k\in 4B_k\cap\partial\Omega$ and $\widetilde{B}_k=B(y_k,4r_k)$. Note that $2B_k\subset 2\widetilde{B}_k$ and $2\widetilde{B}_k\subset 12 B_k$. Let $v_{\widetilde{B}_k}$ be given by (\ref{v_B}). Then we have
\begin{equation}\label{gradu}
	|\nabla^\ell u\chi_{4B_0\cap\Omega}|\leq |\nabla^\ell v_{\widetilde{B}_k}\chi_{4B_0\cap\Omega}|+|\nabla^\ell (u-v_{\widetilde{B}_k})\chi_{4B_0\cap\Omega}|
\end{equation}
on $2B_k$. By a similar argument as in Case 2, we obtain that (\ref{omegaE}) holds. Thus we complete the proof.
\end{proof}

\begin{proof}[\bf Proof of Theorem \ref{main-thm-1}]

Let $u$ be a weak solution to the Dirichlet problem (\ref{DP}) and $B_0=B(x_0,r_0)$ be a ball with the properties that $x_0\in\partial\Omega$ and $0<r_0=c_0\text{diam}(\Omega)$. In view of Lemma \ref{Good-lambda-1.2}, we know that \eqref{good-lambda-1.2} holds.
 Multiplying both sides of \eqref{good-lambda-1.2} with $t^{\frac{2}{q}-1}$ and integrating from $t_0$ to $T$, we obtain
	\begin{equation*}
		\aligned
		\int_{t_0}^T t^{\frac{q}{2}-1}|E(\theta^{-\frac{2}{q}} t)|\,dt
		\le& \frac{\theta}{2}	\int_{t_0}^T t^{\frac{q}{2}-1} |E(t)|\,dt\\
		&+ 	\int_{t_0}^T t^{\frac{q}{2}-1}|\big\{ x\in Q_0: \mathcal{M}_{4B_0} (|f\chi_{4B_0\cap\Omega}|^2 ) (x) >\gamma t \big \}|\,dt.
		\endaligned
	\end{equation*}
	By a change of variable, this gives
	\begin{equation*}
		\aligned
	\frac{\theta}{2}\int_{\theta^{-\frac{2}{q}} t_0}^{\theta^{-\frac{2}{q}}T} t^{\frac{q}{2}-1}|E(t)| \,dt
		\le &\frac{\theta}{2}	\int_{t_0}^{\theta^{-\frac{2}{q}}t_0} t^{\frac{q}{2}-1} |E(t)|\,dt\\
		&+ 	\int_{t_0}^T t^{\frac{q}{2}-1}|\big\{ x\in Q_0: \mathcal{M}_{4B_0} (|f\chi_{4B_0\cap\Omega}|^2 ) (x) >\gamma t \big \}|\,dt,
		\endaligned
	\end{equation*}
	which implies that
	\begin{equation*}
		\aligned
		\int_{0}^{\theta^{-\frac{2}{q}}T} t^{\frac{q}{2}-1}|E(t)|\,dt
		&\leq Ct_0^{\frac{q}{2}}|Q_0| + C\int_{Q_0} \mathcal{M}_{4B_0} (|f\chi_{4B_0\cap\Omega}|^2 )^{\frac{q}{2}}\,dx\\
		&\leq Ct_0^{\frac{q}{2}}|Q_0| + C\int_{4B_0\cap\Omega}  |f|^q \,dx.
		\endaligned
	\end{equation*}
	Letting $T\rightarrow\infty$, we obtain
	\begin{equation*}
		\aligned
		\int_{Q_0}&\big\{ \mathcal{M}_{4B_0} (|\nabla^\ell u\chi_{4B_0\cap\Omega}|^2)  \big\}^{\frac{q}{2}}\,dx\\
		&\qquad\leq C\left(\fint_{4B_0}|\nabla^\ell u\chi_{4B_0\cap\Omega}|^2\right)^{\frac{q}{2}}|Q_0| + C\int_{4B_0\cap\Omega}  |f|^q \,dx.
		\endaligned
	\end{equation*}
Next, since $|\nabla^\ell u\chi_{4B_0\cap\Omega}|^{2}(x)\leq\mathcal{M}_{4B_0} (|\nabla^\ell u\chi_{4B_0\cap\Omega}|^{2})(x) $ a.e.,  by H\"older's inequality and the energy estimate, we obtain
\begin{equation*}
\int_{ Q_0}\
|\nabla^\ell u\chi_{4B_0\cap\Omega}|^q \, dx
 \le C\int_{ 4B_0\cap \Omega}
|f|^q \, dx
+C\left(\fint_{4B_0}|\nabla^\ell u\chi_{4B_0\cap\Omega}|^2\right)^{\frac{q}{2}}|Q_0|\leq C\int_{\Omega}
|f|^q\, dx.
\end{equation*}
By covering $B_0$ with cubes, we see that
\begin{align*}
\int_{ B_0\cap\Omega}|\nabla^\ell u|^q\, dx\leq C\int_\Omega|f|^q\, dx.
\end{align*}
Similarly, we can prove that this inequality also holds for $4B_0\subset\Omega$.
The desired estimate (\ref{w-L2-ineq}) now follows by a simple covering argument. Thus we complete the proof of necessity.
\end{proof}

\subsection{A simpler and stronger sufficient condition}
The next theorem gives a much simpler condition which implies \eqref{WRH}.
\begin{thm}\label{main-thm-2}
	Let $\mathcal{L}(D)$ be a system of elliptic operators of order $2\ell$ given by \eqref{elliptic operator} and $\Omega$  a bounded Lipschitz domain in $\mathbb{R}^d, d\geq 2$.
	Assume that  $A$ is a matrix satisfying \eqref{symmetric} and \eqref{ellipticity}.
Suppose that there exist constants $C_0>0$, $R_0>0$ and $2\leq \lambda< d+2$ such that if $0<r<R<R_0$
	\begin{equation}\label{WWRH}
		\int_{D_r} |\nabla^\ell v|^2\, dx
		\le C \bigg(\frac{r}{R}\bigg)^{\lambda-2}\int_{D_{R}}
		|\nabla^\ell v|^2\, dx
	\end{equation}
holds whenever
   \begin{equation}\label{}
		\mathcal{L}(D)v=0 \quad \text{in}\quad D_{4r} ~~\text{and} \\
		~~~~~D^\alpha v=0 \quad \text{on}~~\Delta_{4r}
	\end{equation}
for all $|\alpha|\leq \ell-1$. Then if
	\begin{equation}\label{p}
		2<p<2+\frac{2}{d+2-\lambda},
	\end{equation}	
the $L^p$ Dirichlet problem \eqref{DP} is uniquely solvable.
\end{thm}

In this section we give the proof of Theorem \ref{main-thm-2}, in view of Theorem \ref{main-thm-1}, it suffices  to show that the weak reverse H\"older inequality
\begin{equation}
\left\{\dashint_{D_r}|\nabla^\ell u|^p \right\}^{\frac{1}{p}} \leqslant C
\left\{\dashint_{D_{2r}}|\nabla^\ell u|^2\right\}^{\frac{1}{2}}
\end{equation}
holds for $2<p<2+\frac{2}{d+2-\lambda}$ if $d\geqslant 4$
whenever $u$
satisfies
   \begin{equation}\label{}
		\mathcal{L}(D)u=0 \quad \text{in}\quad D_{4r} ~~\text{and} \\
		~~~~~D^\alpha u=0 \quad \text{on}~~\Delta_{4r}
	\end{equation}
for all $|\alpha|\leq \ell-1$.

\begin{proof}[Proof of Theorem \ref{main-thm-2}]
For $x=(x^\prime,x_d)\in D_{2r}$, let $\delta(x)=|x_d-\psi(x^\prime)|$. It follows from the interior estimate that
 \begin{equation}\label{interior}
		|D^\alpha u(x)|\leq \frac{C_\alpha}{r^{d+|\alpha|}}\int_{B(x,r)}|u(y)|dy\quad~~for~~any~~B(x,2r)\subset \Omega.
	\end{equation}
This, together with \eqref{WWRH} and Poincar\'e inequality, gives
\begin{align}\label{6.3}
|\nabla^\ell u(x)|&\leq C
\left\{\dashint_{B(x,\delta(x))}|\nabla^\ell u|^2dy\right\}^{\frac{1}{2}}\\\nonumber
&\leq C\delta(x)^{-d/2}r^{d/2}\bigg(\frac{\delta(x)}{r}\bigg)^{\frac{\lambda-2}{2}}
\left\{\dashint_{D_{2r}}|\nabla^\ell
u|^2dy\right\}^{\frac{1}{2}}\\\nonumber
&=C\bigg(\frac{\delta(x)}{r}\bigg)^{\frac{\lambda-2-d}{2}}
\left\{\dashint_{D_{2r}}|\nabla^\ell
u|^2dy\right\}^{\frac{1}{2}}.
\end{align}

Since $A$ is a constant matrix satisfying the conditions $(\ref{ellipticity})$ and $(\ref{symmetric})$, it follows from the $L^2$ estimate (see\cite{PV-1995}\cite{Verchota-1990}) we obtain that
\begin{equation}\label{6.4}
\|(\nabla^\ell u)^*\|_{L^{2}(\partial D_{tr})}\leqslant C\|\nabla_{tan}\nabla^{\ell-1}u\|_{L^{2}(\partial D_{tr})}.
\end{equation}
Integrating with respect to $t\in(1,2)$ on $D_{tr}$, we obtain
\begin{equation}\label{6.5}
\int_{\Delta_r}|(\nabla^\ell u)^*_r|^{2}d\sigma \leqslant \frac{C}{r}\int_{D_{2r}}|\nabla^\ell u|^2dx,
\end{equation}
here $(\nabla^\ell u)^*_r(x^\prime, \psi(x^\prime))$ denotes the nontangential maximal function of $\nabla^\ell u$ with respect
to the Lipschitz sub-domain $D_r$.

Note that if $(p-2)(\frac{\lambda-2-d}{2})>-1$, we have
\begin{align}\label{6.6}
\int_{D_r}|\nabla^\ell u|^pdx&\leqslant \frac{C}{r^{(p-2)(\frac{\lambda-d-2}{2})}}
\int_{D_r}|\nabla^\ell u|^2(\delta(x))^{(p-2)(\frac{\lambda-d-2}{2})}dx\left\{\dashint_{D_{2r}}|\nabla^\ell u|^2dx\right\}^{\frac{p}{2}-1}\nonumber
\\&\leqslant\frac{C}{r^{(p-2)(\frac{\lambda-2-d}{2})}}\int_{\Delta_r}|(\nabla^\ell u)_r^*|^2d\sigma\int_0^{cr}t^{(p-2)(\frac{\lambda-2-d}{2})}dt\left\{\dashint_{D_{2r}}|\nabla^\ell
u|^2dy\right\}^{\frac{p}{2}-1}\nonumber
\\&\leqslant Cr\int_{\Delta_r}|(\nabla^\ell u)_r^*|^2d\sigma
\left\{\dashint_{D_{2r}}|\nabla^\ell
u|^2dy\right\}^{\frac{p}{2}-1},
\end{align}
where (\ref{6.3}) was used in the first inequality.

Combining (\ref{6.5}) and (\ref{6.6}), we obtain the desired estimate
\begin{equation}\label{6.8}
\left\{\dashint_{D_r}|\nabla^\ell
u|^p dx \right\}^{\frac{1}{p}} \leqslant
C\left\{\dashint_{D_{2r}}|\nabla^\ell
u|^2 dx \right\}^{\frac{1}{2}}
\end{equation}
for
$$
2<p<2+\frac{2}{d+2-\lambda}.
$$
Combining with  Theorem \ref{main-thm-1}, we complete the proof.

\end{proof}

\section{Higher order elliptic equations and the polyharmonic equation}
\subsection{Proof of Theorem \ref{main-thm-3}}
\begin{proof}
We divide the proof into two parts.\\
Step $1$ ($p>2$).
In view of Theorem \ref{main-thm-2}, it suffices for us to show for any elliptic operator $\mathcal{L}(D)$ and for some $\lambda>3$, the condition \eqref{WWRH} holds. To see this, let $v$ be a weak solution of
   \begin{equation}\label{3.1}
		\mathcal{L}(D)v=0 \quad \text{in}\quad D_{4r} ~~\text{and} \\
		~~~~~D^\alpha v=0 \quad \text{on}~~\Delta_{4r}.
	\end{equation}
It follows from H\"older's inequality we have
   \begin{align}\label{3.2}
		\int_{D_r}|\nabla^{\ell} v|^2&\leq C
        r \int_{\Delta_r}|(\nabla^{\ell} v)^*|^2\nonumber\\
        &\leq Cr^{1+(d-1)(1-2/q)} \bigg(\int_{\Delta_{\rho r}}|(\nabla^{\ell} v)^*|^q\bigg)^{2/q},
	\end{align}
where $3/2<\rho< 3$ and $q>2$. Choose $q>2$ so that the regularity estimate
$$
\|(\nabla^\ell v)^*\|_{L^{q}(\partial D_{\rho R})}\leq C\|\nabla_{tan}\nabla^{\ell-1}v\|_{L^{q}(\partial D_{\rho R})}
$$
holds. Hence,
\begin{align}\label{3.3}
		\bigg(\int_{D_r}|\nabla^{\ell} v|^2\bigg)^{q/2}
        \leq Cr^{q/2+(d-1)(q/2-1)} \int_{\Omega\cap\partial D_{\rho R}}|\nabla^{\ell} v|^q.
	\end{align}
Integrating both sides of \eqref{3.3} w.r.t. $\rho\in (3/2,3)$, we obtain
\begin{align}\label{3.4}
		\bigg(\int_{D_r}|\nabla^{\ell} v|^2\bigg)^{q/2}
               \leq CR^{-1}r^{q/2+(d-1)(q/2-1)} \int_{D_{3R}}|\nabla^{\ell} v|^q.
	\end{align}
Using the higher integrability of $v$, one obtains that
\begin{align}\label{3.5}
		\int_{D_r}|\nabla^{\ell} v|^2
               &\leq CR^{-d}r^{d} \bigg(\frac{r}{R}\bigg)^{(1-d)\frac{2}{q}}\int_{D_{4R}}|\nabla^{\ell} v|^2\nonumber\\
               &=C\bigg(\frac{r}{R}\bigg)^{d+(1-d)(\frac{2}{q})}\int_{D_{4R}}|\nabla^{\ell} v|^2.
	\end{align}
This means that estimate \eqref{3.5} implies condition \eqref{WWRH} with
$$
\lambda-2=d+\frac{2}{q}(1-d)=1+(d-1)(1-\frac{2}{q}),
$$
that is,
$\lambda> 3$.

Step $2$ ($p<2$). This follows from a duality argument, we refer the readers to \cite{Geng-2012} with slight modification.
\end{proof}

\begin{remark}
We present another method here to establish the $L^p$ solvability of \eqref{DP}. Indeed, by Theorem \ref{main-thm-1}, it suffices  to show that the weak reverse H\"older inequality
\begin{equation}\label{3.6}
\left\{\dashint_{D_r}|\nabla^\ell u|^p \right\}^{\frac{1}{p}} \leqslant C
\left\{\dashint_{D_{2r}}|\nabla^\ell u|^2\right\}^{\frac{1}{2}}
\end{equation}
holds for $2<p<\frac{2d}{d-1}+\e$ if $d\geq 2$
whenever $u$ satisfies
   \begin{equation*}
		\mathcal{L}(D)u=0 \quad \text{in}\quad D_{4r} ~~\text{and} \\
		~~~~~D^\alpha u=0 \quad \text{on}~~\Delta_{4r}
	\end{equation*}
for all $|\alpha|\leq \ell-1$. The case $\frac{2d}{d+1}-\e<p<2$ follows by a duality argument which is similar as in the proof of Theorem \ref{main-thm-3}.

To show \eqref{3.6}, for a Lipschitz domain $D_{\rho r}$ where $\rho\in (1,2)$, we first note that the Sobolev imbedding
implies that
\begin{equation}\label{3.7}
		\|\nabla^\ell u\|_{L^p(D_{\rho r})}\leq C\|\nabla^\ell u\|_{H^{1/2,2}(D_{\rho r})}.
	\end{equation}
Then we may apply the complex interpolation and the square function estimate to have
\begin{equation}\label{3.8}
		\|\nabla^\ell u\|_{L^p(D_{\rho r})}\leq C\|(\nabla^\ell u)^*\|_{L^2(\partial D_{\rho r})}\leq C\|\nabla^\ell u\|_{L^2(\partial D_{\rho r})},
	\end{equation}
where we have used the $L^2$ estimate. Finally, an integration argument yields the desired estimate.
\end{remark}

\subsection{Proof of Theorem \ref{main-thm-4}}
We are ready to give the proof of Theorem \ref{main-thm-4}
\begin{proof}
In the case $\ell\geq 3$, estimate \eqref{WWRH} holds for some $\lambda-2>d-2\ell$ if $d=2\ell+1$ or $2\ell+2$ \cite{Mazya-Donchev-1983}. And the ranges of $p<2$ follows by a duality argument, we refer the reader to \cite{Geng-2012} with slight modification.
\end{proof}

\section{Biharmonic equations}
In this section, we give the proof of Theorem \ref{main-thm-5}. The key ingredient is to show some weighted positivity on Lipschitz and convex domains for biharmonic equations. We should point out here that the approach we used here is inspired by the work of Mazya \cite{Mazya-1979} and further developed by Shen \cite{Shen-2006-JGA}. Recall that $u_i$, $u_{ij}$ denote $\frac{\partial u}{\partial x_i}$, $\frac{\partial^2u}{\partial x_i\partial x_j}$ respectively. Throughout this section, we will always use the summation convention that the repeated indices are summed from $1$ to $d$ and denote the complement of $\overline{\Omega}$ by ${\overline{\Omega}}^c$.

In many of the arguments that appear in the following sections, it will be necessary
to approximate a given Lipschitz domain $\Omega$ by a sequence of $C^\infty$ domains $\Omega_j$. When dealing with estimate on the approximating domain $\Omega_j$ with constant depending on the Lipschitz character of $\Omega_j$, we want the estimates to be extended to $\Omega$ by a limit argument. To show this, we recall an approximation theorem due to Verchota \cite{Verchota-1984}.
\begin{thm}\label{approxi}
Let $\Omega\subset \mathbb{R}^d$ be a bounded Lipschitz domain. Then there exists an
increasing sequence of $C^\infty$ domains $\Omega_j$, $j = 1, 2, . . .$, with the following properties.
\begin{enumerate}

\item
For all $j$, $\overline{\Omega}_j\subset \Omega$.
\item

There exist a sequence of homeomorphisms $\Lambda_j :\partial\Omega \to \partial\Omega_j $ such that $\sup_{P\in\partial\Omega} |P-\Lambda_j(P)| \to 0$ as $j\to\infty$ and $\Lambda_j\in \Gamma_a(P)$ for all $j$ and all $p\in \partial\Omega_j $, where
$\Gamma_a(P)$ is a family of nontangential approach regions for $\Omega$.

\item
There exists $h\in C_0^\infty(\mathbb{R}^d,\mathbb{R}^d)$ and $c>0$ such that
$$\langle h(P), N(P)\rangle \geq c$$ for all $P\in \partial\Omega_j$ and all $j$.

\item
There exists a finite covering of $\partial\Omega$ by coordinate cylinders so that for each
coordinate cylinder $(Z,\psi)$ in the covering, $10Z\cap \partial\Omega_j$ is given by the graph of a $C^\infty$ function $\psi_j$
$$
10Z\cap \partial\Omega_j=10Z\cap \{(x^\prime, \psi_j(x^\prime): x^\prime\in \mathbb{R}^{d-1})\}
$$
Furthermore, one has $\psi_j\to \psi$ uniformly, $\nabla \psi_j\to \nabla \psi$ a.e. as $j\to \infty$, and $\|\nabla \psi_j\|_\infty\leq \|\nabla \psi\|_\infty$.
\end{enumerate}

\end{thm}

\subsection{Weighted positivity of $\Delta^2$ on Lipschitz domains}

The following Lemma was stated in \cite{Shen-2006-JGA} without proof. We give a proof here for sake of the convenience.
\begin{lemma}\label{lemma-4.3}
Let $\Omega$ be a bounded Lipschitz domain in $\mathbb{R}^d$, $d\geq 5$. Suppose that $u\in C^2(\bar{\Omega})$ and $u=0$, $\nabla u=0$ on $\partial\Omega$. Let $y\in {\overline{\Omega}}^c$ be fixed. Then for any $0<\sigma\leq d-4$ and $d^2+2d\sigma-7\sigma^2-8\sigma>0$, we have
\begin{align}\label{estimate-4.3}
\int_\Omega |u|^2 |x-y|^{-4-\sigma}\leq C\int_\Omega \Delta u \Delta\big(\frac{u}{|x-y|^\sigma}\big).
\end{align}
\end{lemma}

\begin{proof}
It follows from integration by part that
\begin{align}\label{4.6}
\int_\Omega u_{kk}  \big(\frac{\partial u}{\partial |x-y|}\big)|x-y|^{-1-\sigma}&=\int_\Omega u_{kk} \frac{\partial u}{\partial x_i}(x_i-y_i)|x-y|^{-2-\sigma}\nonumber
\\&=-\int_\Omega u_k \bigg(\frac{\partial u}{\partial x_i}(x_i-y_i)|x-y|^{-2-\sigma}\bigg)_k\nonumber
\\&=-\int_\Omega u_k u_{ki}(x_i-y_i)|x-y|^{-2-\sigma}-\int_\Omega u_k u_i\bigg(\frac{x_i-y_i}{|x-y|^{2+\sigma}}\bigg)_k.
\end{align}

Moreover, integration by parts again leads to
\begin{align}\label{4.7}
-\int_\Omega u_k u_{ki}(x_i-y_i)|x-y|^{-2-\sigma}&=\int_\Omega u_k u_{ki}(x_i-y_i)|x-y|^{-2-\sigma}+\int_\Omega u_k u_k\bigg(\frac{x_i-y_i}{|x-y|^{2+\sigma}}\bigg)_i\nonumber
\\&=\int_\Omega u_k u_{ki}(x_i-y_i)|x-y|^{-2-\sigma}+\int_\Omega |\nabla u|^2\bigg(\frac{x_i-y_i}{|x-y|^{2+\sigma}}\bigg)_i,
\end{align}
that is
\begin{align}\label{4.8}
-\int_\Omega u_k u_{ki}(x_i-y_i)|x-y|^{-2-\sigma}=\frac{1}{2}\int_\Omega |\nabla u|^2\bigg(\frac{x_i-y_i}{|x-y|^{2+\sigma}}\bigg)_i
\end{align}
and direct calculation shows that
\begin{align}\label{4.9}
\bigg(\frac{x_i-y_i}{|x-y|^{2+\sigma}}\bigg)_i=(d-\sigma-2)|x-y|^{-\sigma-2}
\end{align}
and
\begin{align}\label{4.10}
\bigg(\frac{x_i-y_i}{|x-y|^{2+\sigma}}\bigg)_j=\delta_{ij}|x-y|^{-\sigma-2}-(\sigma+2)(x_i-y_i)(x_j-y_j)|x-y|^{-\sigma-4}.
\end{align}
Then we have
\begin{align}\label{4.11}
\int_\Omega u_i u_j \bigg(\frac{x_i-y_i}{|x-y|^{2+\sigma}}\bigg)_j=&\int_\Omega u_i u_j \delta_{ij} |x-y|^{-2-\sigma}\nonumber
\\&-(\sigma+2)\int_\Omega u_i u_j (x_i-y_i)(x_j-y_j) |x-y|^{-4-\sigma}.
\end{align}
Observe that
$$
u_i(x_i-y_i)=|x-y|\frac{\partial u}{\partial |x-y|}.
$$
This, together with \eqref{4.6}, \eqref{4.8} and \eqref{4.11}, gives
\begin{align}\label{4.12}
\int_\Omega \Delta u \frac{\partial u}{\partial |x-y|}|x-y|^{-1-\sigma}=&\frac{1}{2}(d-4-\sigma)\int_\Omega |\nabla u|^2 |x-y|^{-2-\sigma}\nonumber
\\&+(\sigma+2)\int_\Omega \bigg|\frac{\partial u}{\partial |x-y|}\bigg|^2 |x-y|^{-2-\sigma}.
\end{align}

Next, observe that $\big|\frac{\partial u}{\partial |x-y|}\big|\leq |\nabla u|$, so by \eqref{4.12} we have
\begin{align}\label{4.13}
(\frac{1}{2}(d-4-\sigma)+\sigma+2)\int_\Omega \bigg|\frac{\partial u}{\partial |x-y|}\bigg|^2|x-y|^{-2-\sigma}\leq \int_\Omega \Delta u \frac{\partial u}{\partial |x-y|}|x-y|^{-1-\sigma}
\end{align}
and splitting $\sigma+1=\sigma/2+\sigma/2+1$ and using H\"older's inequality we have
\begin{align}\label{4.14}
\frac{1}{2}(d+\sigma)\int_\Omega \bigg|\frac{\partial u}{\partial |x-y|}\bigg|^2|x-y|^{-2-\sigma}\leq \bigg(\int_\Omega| \Delta u|^2|x-y|^{-\sigma}\bigg)^{1/2}\bigg(\int_\Omega\bigg|\frac{\partial u}{\partial |x-y|}\bigg|^2|x-y|^{-2-\sigma}\bigg)^{1/2},
\end{align}
which is
\begin{align}\label{4.15}
\frac{1}{4}(d+\sigma)^2\int_\Omega \bigg|\frac{\partial u}{\partial |x-y|}\bigg|^2|x-y|^{-2-\sigma}\leq \int_\Omega| \Delta u|^2|x-y|^{-\sigma}.
\end{align}

To proceed, we claim that
\begin{align}\label{estimate-4.1}
\int_\Omega u_{ii} (u |x-y|^{-\sigma})_{ii}dx=&\int_\Omega |\Delta u|^2|x-y|^{-\sigma}dx+2\sigma\int_\Omega |\nabla u|^2|x-y|^{-2-\sigma}dx\nonumber
\\&-2\sigma(\sigma+2)\int_\Omega \left|\frac{\partial u}{\partial |x-y|}\right|^2|x-y|^{-2-\sigma}dx\nonumber
\\&+\frac{1}{2}\sigma(\sigma+2)(d-\sigma-2)(d-\sigma-4)\int_\Omega |u|^2|x-y|^{-4-\sigma}.
\end{align}
Then it follows from  \eqref{4.15} and the fact $\big|\frac{\partial u}{\partial |x-y|}\big|\leq |\nabla u|$ as well as claim \eqref{estimate-4.1} that
\begin{align}\label{4.16}
\bigg(\frac{1}{4}(d+\sigma)^2+2\sigma-2\sigma(\sigma+2)\bigg)\int_\Omega \bigg|\frac{\partial u}{\partial |x-y|}\bigg|^2|x-y|^{-2-\sigma}\leq \int_\Omega \Delta u \Delta (u|x-y|^{-\sigma}),
\end{align}
i.e.,
\begin{align}\label{4.17}
\frac{1}{4}\bigg(d^2+2d\sigma-7\sigma^2-8\sigma\bigg)\int_\Omega \bigg|\frac{\partial u}{\partial |x-y|}\bigg|^2|x-y|^{-2-\sigma}\leq \int_\Omega \Delta u \Delta (u|x-y|^{-\sigma}).
\end{align}

Utilize the assumption $d^2+2d\sigma-7\sigma^2-8\sigma>0$, we have
 $2\sigma(\sigma+1)\leq \frac{1}{4}(d+\sigma)^2$. With these, together with claim \eqref{estimate-4.1} again,
it is easy to see that the summation of the first three term in the r.h.s of claim \eqref{estimate-4.1} is nonnegative. Hence we obtain that
\begin{align}\label{4.19}
\int_\Omega |u|^2 |x-y|^{-4-\sigma}dx\leq C\int_\Omega \Delta u \Delta\big(\frac{u}{|x-y|^\sigma}\big),
\end{align}
which is the desired estimate \eqref{estimate-4.3}.

Finally, it suffices for us to give the proof of claim \eqref{estimate-4.1}.
Direct computation shows that
\begin{align}\label{4.2}
(u |x-y|^{-\sigma})_{ij}=u_{ij}|x-y|^{-\sigma}+u_i (|x-y|^{-\sigma})_j+u_j (|x-y|^{-\sigma})_i+ u(|x-y|^{-\sigma})_{ij}.
\end{align}
Note that $u=0, \nabla u=0$ on $\partial\Omega$. Then it follows from integration by part we obtain that
\begin{align}\label{4.3}
2\int_\Omega \Delta u u_j (|x-y|^{-\sigma})_j&=-2\int_\Omega u_{i} \big[u_j (|x-y|^{-\sigma})_j\big]_i\\\nonumber
&=-2\int_\Omega u_{i}u_{ij} (|x-y|^{-\sigma})_j-2\int_\Omega u_i u_j \big[|x-y|^{-\sigma}\big]_{ij}.
\end{align}
Use integration by part again yields that
\begin{align}\label{4.4}
-2\int_\Omega u_{i}u_{ij} (|x-y|^{-\sigma})_j=\int_\Omega  |\nabla u|^2\Delta (|x-y|^{-\sigma}).
\end{align}
By the same argument we have
\begin{align}\label{4.5}
\int_\Omega u\Delta u \Delta (|x-y|^{-\sigma})&=-\int_\Omega|\nabla u|^2 \Delta (|x-y|^{-\sigma})+\frac{1}{2}\int_\Omega |u|^2\Delta^2 (|x-y|^{-\sigma}).
\end{align}
In view of \eqref{4.2}-\eqref{4.5}, together with a simple calculation
\begin{align}\label{5.1.5.2-9}
&( |x-y|^{-\sigma})_{ij}=-\sigma|x-y|^{-\sigma-2}\delta_{ij}+\sigma(\sigma+2)(x_i-y_i)(x_j-y_j)|x-y|^{-\sigma-4},\nonumber
\\&\Delta^2(|x-y|^{-\sigma})=\sigma(\sigma+2)(d-2-\sigma)(d-4-\sigma)|x-y|^{-\sigma-4},
\end{align}
lead to the claim \eqref{estimate-4.1}. Thus we complete the proof.
\end{proof}

\begin{lemma}\label{lemma-4.2}
Suppose that $u\in C^2(\bar{\Omega})$ and $u=0$, $\nabla u=0$ on $\partial\Omega$. Let $y\in {\overline{\Omega}}^c$ be fixed. Then for any $0<\sigma\leq d-4$ and $d^2+2d\sigma-7\sigma^2-8\sigma>0$, we have
\begin{align}\label{estimate-4.2}
\int_\Omega |\nabla u|^2 |x-y|^{-2-\sigma}dx\leq C\int_\Omega \Delta u \Delta\big(\frac{u}{|x-y|^\sigma}\big).
\end{align}
\end{lemma}
\begin{proof}
The proof is similar to that of Lemma \ref{lemma-4.3}.
\end{proof}

\subsection{Proof of Theorem \ref{main-thm-5}}
We are ready to give the proof of Theorem \ref{main-thm-5}.

\begin{proof}[Proof of Theorem \ref{main-thm-5}]
 Armed with Theorem \ref{approxi}, we may approximate $\Omega$ by a sequence of smooth domains $\{\Omega_j\}$ with smooth boundaries and let $u_j$ be the solution to the Dirichlet problem \eqref{DP} on $\Omega_j$ with $F_j\in W^{-\ell,p}$. And in the following argument, we will not make effort to distinguish $u_j, \Omega_j$ with $u,\Omega$ correspondingly.

 Note that in the case $\ell=2$, estimate \eqref{WWRH} holds for some $\lambda>d-2\ell+2$ if $d=5,6$ or $7$ (see\cite{Mazya-1979}). Also note that in the case of $d=2,3$ and $4$, the ranges of $p$ follows from Theorem \ref{main-thm-3} directly. Then in view of Theorem \ref{main-thm-2}, it suffices for us to show that if $d\geq 8$, then
estimate \eqref{WWRH} holds for any $\lambda<\lambda_d=\sigma_d+2$ with
$$
\sigma_d=\frac{d-4+2\sqrt{2(d^2-d+2)}}{7}.
$$

Let $0<R<R_0$. For any fixed $y\in {\overline{\Omega}}^c$, by Lemma \ref{lemma-4.3}, we have
\begin{align}\label{4.19}
\int_\Omega |u|^2 |x-y|^{-4-\sigma}dx\leq C\int_\Omega \Delta u \Delta\big(\frac{u}{|x-y|^\sigma}\big).
\end{align}
Next, one may note that \eqref{4.19} also holds for $y\in\partial\Omega$ due to the Lebesgue Dominated Convergence Theorem. Assume that $v$ is a weak solution of $\Delta^2 v=0$ in $\Omega$ and $v=\frac{\partial v}{\partial N}=0$ on $\Delta_R$. Set $u=v\varphi\in W_0^{2,2}(\overline{\Omega})$ where $\varphi$ is a cut-off function such that $supp(\varphi)\subset B(y, 2r)$ with $r<R/4$ satisfying
$\varphi\in C_0^\infty(\mathbb{R}^d)$
	\begin{equation}
		\begin{cases}
			\varphi=1 \quad \text{ in } B(y, r),\\
			\varphi=0 \quad \text{ outside } B(y, 2r),\\
			|\nabla^k\varphi|\le C/r^k.
		\end{cases}
	\end{equation}
Then by \eqref{4.19} we obtain
\begin{align}\label{4.20}
\int_\Omega |v\varphi|^2 |x-y|^{-4-\sigma}dx\leq C\int_\Omega \Delta (v\varphi) \Delta\big(\frac{v\varphi}{|x-y|^\sigma}\big).
\end{align}
Notice that $\supp {|\nabla \varphi|}\subset \{x\in\mathbb{R}^d: r\leq |x-y|\leq 2r\}$. This, together with a limiting argument and the boundary Cacciopoli's inequality \cite{Shen-2006-JGA}
\begin{align}\label{Cacciopoli}
\frac{1}{r^2}\int_{D_r} |\nabla u|^2 dx+\int_{D_r} |\nabla^2 u|^2dx\leq \frac{C}{r^4}\int_{D_{2r}\backslash D_{r}} |u|^2dx,
\end{align}
leads to that
\begin{align}\label{4.21}
\int_{D_{r}} |v|^2 dx\leq C\bigg(\frac{r}{R}\bigg)^{\sigma+4+\delta}\int_{D_{R}} |v|^2dx.
\end{align}
Utilize Cacciopoli's inequality (with $\ell=1$)
\begin{align}\label{Cacciopoli-1}
\int_{D_{r}} |\nabla^\ell v|^2 dx\leq \frac{C}{r^2}\int_{D_{2r}} |\nabla^{\ell-1}v|^2dx
\end{align}
and Poincar\'e inequality again yields the desired estimate
\begin{align}\label{4.22}
\int_{D_{r/2}} |\nabla v|^2 dx\leq C\bigg(\frac{r}{R}\bigg)^{\sigma+2+\delta}\int_{D_R} |\nabla v|^2dx,
\end{align}
we omit the detail and refer the reader to \cite{Shen-2006-JGA}. Finally, we have established the desired estimate \eqref{WWRH} for $\lambda-2=\sigma+\delta=\lambda_d-2+\delta$, i.e., $\lambda=\lambda_d+\delta$.

We utilize Theorem \ref{main-thm-1} to show second part of Theorem \ref{main-thm-5}. To handle this,
it suffices for us to show that
estimate \eqref{WRH} holds for any $p>2$. Let $\Omega$ be convex domain in $\mathbb{R}^d, d\geq 2$ and $u$ a weak solution of $\Delta^2 u=0$ in $\Omega$ and $u=\frac{\partial u}{\partial N}=0$ on $\Delta_{5r}$.
It follows from Theorem 1.1 in \cite{Mazya-Mayboroda-2008} that
$$
\sup_{B(P,r)\cap\Omega} |\nabla^2 u|
\le \frac{C}{r^2} \left\{\frac{1}{r^d}
\int_{B(P,3r)\cap \Omega}|u|^2\, dx \right\}^{1/2}.
$$
This, together with $u=|\nabla u|=0$ on $B(P,5r)\cap \partial\Omega$
$P\in \partial\Omega$ and $0<r<r_0$.,
 gives that
\begin{equation}
\label{reverseHolder}
\sup_{B(P,r)\cap \Omega} |\nabla^2 u|
\le C \left\{ \frac{1}{r^{d}}
\int_{B(P,3r)\cap \partial\Omega}
|\nabla^2 u|^2\, d\sigma\right\}^{1/2}.
\end{equation}
This gives estimate \eqref{WRH} immediately. Thus we complete the proof.
\end{proof}

\section{Counterexample to Unique Solvability of the Dirichlet Porblem}
In this section, we will show that in bounded Lipscthiz domain $d>7$ and $p>5/2+\e$ that the Dirichlet problem fails to be solvable for biharmonic equation, and for $d>2\ell+2$ and a $p>2+\frac{1}{\ell}+\e$, the Dirichlet problem fails to be solvable for polyharmonic equation.

Our aim is to prove the following theorems.

\begin{thm}\label{Counterexample-2}
Given an integer $\ell\geq 3$ and a dimension $d, 4\leq d\leq 2\ell+2$ and a $p>2+\frac{1}{\ell}+\e$ there is a bounded Lipschitz domain $\Omega\subset\mathbb{R}^d$ so that
the polyharmonic Dirichlet problem in the sense of Theorem \ref{main-thm-4} can not be uniquely solved with data in $W^{-\ell,p}(\partial\Omega)$.
\end{thm}
\begin{proof}
In view of Theorem \ref{main-thm-4}, we known that the polyharmonic Dirichlet problem $\Delta^\ell u=F$ in $\Omega$, $\ell\geq 3$ is uniquely solvable for $2-\frac{1}{\ell+1}-\e<p< 2+\frac{1}{\ell}+\e$ if $2\leq d\leq  2\ell+2$. To deal with this, it suffices for us to show that the optimal ranges of $p$ is
$$
2-\frac{1}{\ell+1}-\e<p< 2+\frac{1}{\ell}+\e~~ \text{if}~~ d=2\ell+1
$$
 for solution to the polyharmonic equation.

We follow \cite{MNP-1985}(see also \cite{PV-1995-1}) to show the sharpness of ranges of $p$. To handle this, let
$$
\Gamma_\e=\{x\in\mathbb{R}^d:(x_1^2+x_2^2+...+x_{d-1}^2)<-\e x_d\},~~\e>0
$$
and $\omega_\e=\mathbb{R}^d\backslash \Gamma_\e$. The failure of the $W^{\ell,p}$ solvability on Lipschitz domains $\Omega\subset \mathbb{R}^d$, $d>2\ell+1$ and $p> 5/2+\e $,
follows from the existence of a polyharmonic function $v(x)$ in the exterior of a cone $\Gamma_\e$, satisfying
\begin{align*}
&(a)~~~u(x)=|x|^\lambda \phi\bigg(\frac{x}{|x|}\bigg),~~~\lambda~~\text{is}~~\text{real},\\
&(b)~~~ \frac{\partial^{\ell-1}u }{\partial N^{\ell-1}}=0~~and~~D^\alpha u=0 ~~\text{for}~~ |\alpha|\leq \ell-2~~\text{on the lateral sides of the cone } \Gamma_\e,\\
&(c)~~ \text{If the aperture of the cone is small enough, then}~~ \lambda\text{ may be chosen to be less than $1$}.
\end{align*}
Notice that the existence of such a $u$ was established in \cite{MNP-1985} if $d>5$ and shown
in \cite{PV-1992} if $d=4$. We also notice that in the case of biharmonic equation, the existence of such a $u$ was established in \cite{MP-1981}.

Let $\sigma=(\sigma_1,...,\sigma_d)$ denote points on the unit sphere $\mathbb{S}^{d-1}$. Upon applying $\Delta^\ell$ to $v(x)$ leads to the eigenvalue problems in spherical domains $\Omega_\e=\mathbb{S}^{d-1}\cap \omega_\e$
\begin{equation}\label{eigenvalue}
	\left\{
	\aligned
	\prod_{j=1}^\ell (\Lambda_j+\Delta_\sigma)u  & =0 & \quad & \text{ in } \Omega_\e,\\
	D^\alpha_\sigma u & =0 & \quad & \text{ on } \partial\Omega_\e, ~~\text{for}~~|\alpha|\leq \ell-1,
	\endaligned
	\right.
\end{equation}
where $\Delta_\sigma$ denote the spherical Laplacian and $\Lambda_j=(\lambda-2j+2)(\lambda-2j+d)$ for $j=1,...,\ell.$

Next, it follows from Theorem $2.6$ in \cite{PV-1995-1} that for any given $\ell\geq 2$ and $2\leq d\leq 2\ell+1$, the eigenvalue problem \eqref{eigenvalue} has nontrivial solution $u\in C_0^\infty(\bar{\Omega}_\e)$ with
 $$
 \lambda(\e)\to \ell-\frac{d}{2}+\frac{1}{2}~~\text{as}~~\e\to 0.
 $$
Thus we know that in the case $d=2\ell+1$, then $\lambda(\e)\to 0$ as $\e\to 0$. This, together with a simple calculation shows that
$$
(\lambda-\ell)p+2\ell+1>0
$$
which yields that
$$
p<2+\frac{1}{\ell}.
$$
This, together with the self-improving property of the weak reverse H\"older inequality and a duality argument, yields that the $W^{\ell,p}$ solvability on Lipschitz domains $\Omega\subset \mathbb{R}^d$ is sharp for $2-\frac{1}{\ell+1}-\e<p< 2+\frac{1}{\ell}+\e$ if $2\leq d\leq 2\ell+2$.
Thus we complete the proof.
\end{proof}

\begin{thm}\label{Counterexample-1}
Given any dimension $d > 7$ and a $p>5/2+\e$ there is a bounded Lipschitz domain $\Omega\subset\mathbb{R}^d$ so that
the biharmonic Dirichlet problem in the sense of Theorem \ref{main-thm-5} can not be uniquely solved with data in $W^{-\ell,p}(\partial\Omega)$.
\end{thm}

\begin{proof}
The proof follows the same manner as in the proof of Theorem \ref{Counterexample-2} with slight modification. In view of Theorem \ref{main-thm-5}, we known that the the biharmonic Dirichlet problem $\Delta^2 u=F$ in $\Omega$ is uniquely solvable for $5/3-\e<p< 5/2+\e $ if $d=5,6,7$. Hence it suffices for us to show the optimal ranges of $p$ is
$$\frac{5}{3}-\e<p< \frac{5}{2}+\e ~~\text{if}~~d=5$$ for solution to biharmonic equation.

To do this, we follows from \cite{MP-1981} to construct the counterexample. Let $\Gamma_\e$ be the same as in Theorem \ref{Counterexample-2} and $u$ a biharmonic function $u(x)$ in the exterior of a cone $\Gamma_\e(0)$, with vertex at the origin satisfying
$u(x)=|x|^\lambda \phi\bigg(\frac{x}{|x|}\bigg),$
$u=0$ and $\frac{\partial u}{\partial N}=0 $ on the lateral sides of the cone $ \Gamma_\e(0).$
A bounded domain $\tilde{\Omega}$ may be constructed so that $\partial\tilde{\Omega}$ is smooth except at the
origin. Set
$$
B(0,1/2)\backslash \Gamma_\e(0)\subset \tilde{\Omega}\subset B(0,1)\backslash \Gamma_\e(0)
$$
with $\partial\tilde{\Omega}\cap\partial \Gamma(0)=\partial \Gamma_\e(0)\cap B(0,1/2)$.
By the smoothness of $\partial\tilde{\Omega}$ away from the origin and interior estimates, we have that $|\nabla u(x)|\in L^\infty(\partial\tilde{\Omega})$, thus $|\nabla^2 u|=O(|x|^{\lambda-2})$ near the origin. Next, it follows from Theorem $2.6$ in \cite{PV-1995-1}, we know that in the case $d=5$ and $\ell=2$, then $\lambda(\e)\to 0$ as $\e\to 0$. This, together with a simple calculation shows that
$$(\lambda-2)p+5>0$$
which yields that
$$p<\frac{5}{2}.$$
 This, together with the self-improving property of the weak reverse H\"older inequality and a duality argument, yields the failure of the $W^{\ell,p}$ solvability on Lipschitz domains $\Omega\subset \mathbb{R}^d$, $d>5$ and $5/3-\e<p< 5/2+\e$.
Thus we complete the proof.
\end{proof}

\medskip

\bigskip

\newpage

 \bibliographystyle{amsplain}

\bibliography{JGeng-2025}

\providecommand{\bysame}{\leavevmode\hbox to3em{\hrulefill}\thinspace}
\providecommand{\MR}{\relax\ifhmode\unskip\space\fi MR }
\providecommand{\MRhref}[2]{%
  \href{http://www.ams.org/mathscinet-getitem?mr=#1}{#2}
}
\providecommand{\href}[2]{#2}
\begin{thebibliography}{10}

\bibitem{AP-1998}
V.~Adolfsson and J.~Pipher, \emph{The inhomogeneous {D}irichlet problem for
  ${\Delta^2}$ in {L}ipschitz domains}, J. Funct. Anal. \textbf{159} (1998),
  no.~1, 137--190.

\bibitem{AQ-2002}
P.~Auscher and M.~Qafsaoui, \emph{Observations on ${W}^{1,p}$ estimates for
  divergence elliptic equations with {VMO} coefficients}, Boll. Unione Mat.
  Ital. Sez. B Artic. Ric. Mat. (8) \textbf{5} (2002), no.~2, 487--509.

\bibitem{BHM-2018}
A.~Barton, S.~Hofmann, and S.~Mayboroda, \emph{The {N}eumann problem for higher
  order elliptic equations with symmetric coefficients}, Math. Ann.
  \textbf{371} (2018), no.~1-2, 297--336.

\bibitem{BHM-2019}
\bysame, \emph{Dirichlet and {N}eumann boundary values of solutions to higher
  order elliptic equations}, Ann. Inst. Fourier (Grenoble) \textbf{69} (2019),
  no.~4, 1627--1678.

\bibitem{BHM-2021}
\bysame, \emph{Nontangential estimates on layer potentials and the {N}eumann
  problem for higher-order elliptic equations}, Int. Math. Res. Not. IMRN
  (2021), no.~23, 18300--18366.

\bibitem{DKV-1986}
B.~Dahlberg, C.~Kenig, and G.~Verchota, \emph{The {D}irichlet problem for the
  biharmonic equation in a {L}ipschitz domain}, Ann. Inst. Fourier (Grenoble)
  \textbf{36} (1986), no.~3, 109--135.

\bibitem{DKV-1988}
\bysame, \emph{Boundary value problems for the systems of elastostatics in
  {L}ipschitz domains}, Duke Math. J. \textbf{57} (1988), no.~3, 795--818.

\bibitem{Geng-2010}
J.~Geng, \emph{The {N}eumann problem and {H}elmholtz decomposition in convex
  domains}, J. Funct. Anal. \textbf{259} (2010), no.~8, 2147--2164.

\bibitem{Geng-2012}
\bysame, \emph{${W}^{1,p}$ estimates for elliptic problems with {N}eumann
  boundary conditions in {L}ipschitz domains}, Adv. Math. \textbf{229} (2012),
  no.~4, 2427--2448.

\bibitem{JK-1995}
D.~Jerison and C.~Kenig, \emph{The inhomogeneous {D}irichlet problem in
  {L}ipschitz domains}, J. Funct. Anal. \textbf{130} (1995), no.~1, 161--219.

\bibitem{KS-2011-MA}
J.~Kilty and Z.~Shen, \emph{A bilinear estimate for biharmonic functions in
  {L}ipschitz domains}, Math. Ann. \textbf{349} (2011), no.~2, 367--394.

\bibitem{Mazya-Mayboroda-2008}
S.~Mayboroda and V.~Maz'ya, \emph{Boundedness of the {H}essian of a biharmonic
  function in a convex domain}, Comm. Partial Differential Equations
  \textbf{33} (2014), no.~7-9, 1439--1454.

\bibitem{Mazya-1979}
V.~Maz'ya, \emph{Behaviour of solutions to the {D}irichlet problem for the
  biharmonic operator at a boundary point}, Equadiff IV (Proc. Czechoslovak
  Conf. Differential Equations and their Applications, Prague, 1977), vol. 703,
  Springer, Berlin, pp.250-262. Lecture Notes in Math., 703, 1979.

\bibitem{Mazya-Donchev-1983}
V.~Maz'ya and T.~Donchev, \emph{Regularity in the sense of wiener of a boundary
  point for a polyharmonic operator. (russian)}, C. R. Acad. Bulgare Sci.
  \textbf{36} (1983), no.~2, 177--179.

\bibitem{MMS-2010}
V.~Maz'ya, M.~Mitrea, and T.~Shaposhnikova, \emph{The {D}irichlet problem in
  {L}ipschitz domains for higher order elliptic systems with rough
  coefficients}, J. Anal. Math. \textbf{110} (2010), 167--239.

\bibitem{MNP-1985}
V.~Maz'ya, S.~A. Nazarov, and B.~A. Plamenevski$\breve{l}$, \emph{On the
  singularities of solutions of the dirichlet problem in the exterior of a
  slender cone}, Math. USSR Sbornik \textbf{50} (1985), no.~2, 52--59.

\bibitem{MP-1981}
V.~Maz'ya and B.~A. Plamenevski$\breve{l}$, \emph{On the maximum principle for
  the biharmonic equation in a domain with conical points. (russian)}, Izv.
  Vyssh. Uchebn. Zaved. Mat. (1981), no.~2, 52--59.

\bibitem{MM-2013-1}
I.~Mitrea and M.~Mitrea, \emph{Boundary value problems and integral operators
  for the bi-{L}aplacian in non-smooth domains}, Atti Accad. Naz. Lincei Rend.
  Lincei Mat. Appl. \textbf{24} (2013), no.~3, 329--383.

\bibitem{MM-2013}
\bysame, \emph{Multi-layer potentials and boundary problems for higher-order
  elliptic systems in {L}ipschitz domains}, Lecture Notes in Mathematics, 2063,
  vol. 249, Springer, Heidelberg, x+424 pp., 2013.

\bibitem{MMW-2011}
I.~Mitrea, M.~Mitrea, and M.~Wright, \emph{Optimal estimates for the
  inhomogeneous problem for the bi-{L}aplacian in three-dimensional {L}ipschitz
  domains. {P}roblems in mathematical analysis. no. 51}, J. Math. Sci. (N.Y.)
  \textbf{172} (2011), no.~1, 24--134.

\bibitem{PV-1992}
J.~Pipher and G.~Verchota, \emph{The {D}irichlet problem in ${L}^p$ for the
  biharmonic equation on {L}ipschitz domains}, Amer. J. Math. \textbf{114}
  (1992), no.~5, 923--972.

\bibitem{PV-1993}
\bysame, \emph{A maximum principle for biharmonic functions in {L}ipschitz and
  ${C}^1$ domains}, Comment. Math. Helv. \textbf{68} (1993), no.~3, 385--414.

\bibitem{PV-1995}
\bysame, \emph{Dilation invariant estimates and the boundary
  {G}$\mathring{a}$rding inequality for higher order elliptic operators}, Ann.
  of Math. \textbf{142} (1995), no.~1, 1--38.

\bibitem{PV-1995-1}
\bysame, \emph{Maximum principles for the polyharmonic equation on lipschitz
  domains.}, Potential Anal. \textbf{4} (1995), no.~6, 615--636.

\bibitem{Shen-2005}
Z.~Shen, \emph{Bounds of {R}iesz transforms on ${L}^p$ spaces for second order
  elliptic operators}, Ann. Inst. Fourier (Grenoble) \textbf{55} (2005), no.~1,
  173--197.

\bibitem{Shen-2006}
\bysame, \emph{The ${L}^p$ {D}irichlet problem for elliptic systems on
  {L}ipschitz domains}, Math. Res. Lett. \textbf{13} (2006), 143--159.

\bibitem{Shen-2006-MA}
\bysame, \emph{Necessary and sufficient conditions for the solvability of the
  ${L}^p$ {D}irichlet problem on {L}ipschitz domains}, Math. Ann. \textbf{336}
  (2006), no.~3, 697--725.

\bibitem{Shen-2006-JGA}
\bysame, \emph{On estimates of biharmonic functions on {L}ipschitz and convex
  domains}, J. Geom. Anal. \textbf{16} (2006), no.~4, 721--734.

\bibitem{Shen-2007}
\bysame, \emph{The ${L}^p$ boundary value problems on {L}ipschitz domains},
  Adv. Math. \textbf{216} (2007), no.~1, 212--254.

\bibitem{Verchota-1984}
G.~Verchota, \emph{Layer potentials and regularity for the {D}irichlet problem
  for {L}aplace's equation in {L}ipschitz domains}, J. Funct. Anal. \textbf{59}
  (1984), no.~3, 572--611.

\bibitem{Verchota-1987}
\bysame, \emph{The {D}irichlet problem for the biharmonic equation in ${C}^1$
  domains}, Indiana Univ. Math. J. \textbf{36} (1987), no.~4, 867--895.

\bibitem{Verchota-1990}
\bysame, \emph{The {D}irichlet problem for the polyharmonic equation in
  {L}ipschitz domains}, Indiana Univ. Math. J. \textbf{39} (1990), no.~3,
  671--702.

\end{thebibliography}

\bigskip

\bigskip

\end{document}